\documentclass{amsart}

\usepackage{amsmath}
\usepackage{amssymb}
\usepackage{amsthm}
\usepackage{esint}
\usepackage[inline]{enumitem}
\usepackage{color}
\usepackage[dvipdfmx]{hyperref}
\hypersetup{colorlinks=true, citecolor=blue, linkcolor=red, urlcolor=blue}

\theoremstyle{plain}
\newtheorem{theorem}{Theorem}[section]
\newtheorem{corollary}[theorem]{Corollary}
\newtheorem{lemma}[theorem]{Lemma}
\newtheorem{proposition}[theorem]{Proposition}

\theoremstyle{definition}

\newtheorem{remark}[theorem]{Remark}

\theoremstyle{remark}

\numberwithin{theorem}{section}
\numberwithin{equation}{section}

\newcommand{\Z}{\mathbb{Z}}

\newcommand{\R}{\mathbb{R}}

\newcommand{\dist}{\mathrm{dist}}
\newcommand{\diam}{\mathrm{diam}}

\newcommand{\cl}{\overline}

\newcommand{\loc}{\mathrm{loc}}

\DeclareMathOperator*{\divergence}{div}

\newcommand{\laplacian}{\Delta}

\DeclareMathOperator*{\osc}{osc}
\DeclareMathOperator*{\esssup}{ess\,sup}
\DeclareMathOperator*{\essinf}{ess\,inf}

\newcommand{\capacity}{\mathrm{cap}}

\newcommand{\wkto}{\rightharpoonup}

\begin{document}


\title{Strong barriers for weighted quasilinear equations}
\author{Takanobu Hara}
\email{takanobu.hara.math@gmail.com}
\address{Graduate School of Information Science and Technology, Kita 14, Nishi 9, Kita-ku, Sapporo, Hokkaido, 060-0814, Japan}
\date{\today}


\begin{abstract}
In potential theory, use of barriers is one of the most important techniques.
We construct strong barriers for weighted quasilinear elliptic operators.
There are two applications: (i) solvability of Poisson-type equations with boundary singular data,
and (ii) a geometric version of Hardy inequality.
Our construction method can be applied to a general class of divergence form elliptic operators on domains with rough boundary.
\end{abstract}



\maketitle


\section{Introduction}\label{sec:introduction}

Let $\Omega \subsetneq \R^{n}$ ($n \ge 1$) be an open set with nonempty boundary, and let $1 < p < \infty$.
We consider elliptic differential equations of the type
\begin{equation}\label{eqn:p-laplace}
\begin{cases}
- \divergence \mathcal{A}(x, \nabla u(x)) = f(x) & \text{in} \ \Omega, \\
u = 0 & \text{on} \ \partial \Omega,
\end{cases}
\end{equation}
where 
$\divergence \mathcal{A}(x, \nabla \cdot)$ is a weighted $(p, w)$-Laplacian type elliptic operator,
$w$ is a doubling weight on $\R^{n}$ which admit a $p$-Poincar\'{e} inequality (see \eqref{eqn:VD} and \eqref{eqn:PI} for detail),
and $f$ is a locally bounded function on $\Omega$.
The most simple example of $w$ is $w \equiv 1$, and the reason for considering weighted equations will be explained later.
The precise assumptions on $\mathcal{A} \colon \Omega \times \R^{n} \to \R^{n}$ are as follows:
For each $z \in \R^{n}$, $\mathcal{A}(\cdot, z)$ is measurable, for each $x \in \Omega$, $\mathcal{A}(x, \cdot)$ is continuous,
and there exists $1 \le L < \infty$ such that
\begin{align}
\mathcal{A}(x, z) \cdot z \ge w(x) |z|^{p}, \label{eqn:coercive}
\\ 
|\mathcal{A}(x, z)| \le L w(x) |z|^{p - 1},  \label{eqn:growth}
\\
\left( \mathcal{A}(x, z_{1}) - \mathcal{A}(x, z_{2}) \right) \cdot (z_{1} - z_{2}) > 0, \label{eqn:monotonicity}
\\ 
\mathcal{A}(x, t z) = t |t|^{p - 2} \mathcal{A}(x, z) \label{eqn:homogenity}
\end{align}
for all $x \in \Omega$, $z, z_{1}, z_{2} \in \R^{n}$, $z_{1} \not = z_{2}$ and $t \in \R$.
When $\mathcal{A}(x, z) = w(x)|z|^{p - 2} z$, the operator $\divergence \mathcal{A}(x, \nabla u)$ is called as the $(p, w)$-Laplacian.
In particular, if $\mathcal{A}(x, z) = z$, then the operator coincides with the classical Laplacian.

The aim of this paper is to provide an existence result of weak solutions to \eqref{eqn:p-laplace} for boundary singular data.
The study of Eq. \eqref{eqn:p-laplace} has long history,
and the standard approach to solve this problem is the variational method and its generalization (see, \cite{MR0259693}).
However, this method yields only solutions with finite energy.
On the other hand, we can confirm existence of infinite energy solutions for boundary singular data
by considering Poisson's equation (in the classical sense) or ordinary differential equations.
To obtain such solutions, we divide the problem into three steps:
\begin{enumerate*}[label=(\roman*)]
\item\label{item:first_step}
find finite energy solutions to approximate problems,
\item\label{item:second_step}
prove compactness of solutions, and
\item\label{item:third_step}
derive a uniform bound for solutions.
\end{enumerate*}
The part \ref{item:first_step} is trivial in our setting,
and we can find sufficient results for \ref{item:second_step} in prior work
(e.g., \cite{MR1183665, MR2305115, MR1890997}).
This sequel from \cite{MR4309811, MR4493271} proposes a new perspective for \ref{item:third_step}.
More precisely, we construct supersolutions called strong barriers.

Supersolutions are an effective tool to control boundary behavior of solutions.
For concrete problems, direct calculations for $\delta$ often yield sharp estimates 
(see, e.g., \cite{MR700735, MR1817710, MR2286361, MR2286038, MR2297247}).
Unfortunately, this method can not be applied to general elliptic equations on domains with rough boundary.
When conditions of the form \eqref{eqn:coercive}-\eqref{eqn:homogenity} have to be considered,
such as in applications to homogenization problems (e.g., \cite{MR503330}),
it is needed to consider a construction method of supersolutions itself. 

Ancona \cite{MR856511} defined strong barriers for linear elliptic operators 
and constructed them under capacity density type conditions.
In addition, a Hardy inequality was proved as one application of them.
Since other proofs by Lewis \cite{MR946438} and Wannebo \cite{MR1010807}, 
many authors have proved more general Hardy-type inequalities under capacity density conditions (see, \S \ref{sec:setting}).
However, another application of strong barriers, to the Dirichlet problem of the type \eqref{eqn:p-laplace},
seems not to be discussed sufficiently.

We construct strong barriers for quasilinear operators (Theorem \ref{thm:main})
and apply the result to \eqref{eqn:p-laplace} (Theorem \ref{thm:Dirichlet}).
Specifically, we make auxiliary functions by a boundary H\"{o}lder estimate in the De Giorgi-Nash-Moser theory
and  construct a global function by gluing them.
These results can be regarded as extensions of \cite[Remark 6.2]{MR856511} or \cite[Corollary 4.4]{MR4493271}
as well as analogs of known facts for the $p$-Laplacian on $C^{2}$ domains
(see, \cite[Problem 6.6]{MR1814364}, \cite[Theorem 1]{MR1817710}).
In addition, we prove a Hardy-type inequality (Corollary \ref{cor:hardy})
by combining the results and the Picone inequality (see \cite{MR1618334}).
For connection with prior work on Hardy-type inequalities, we consider weighted operators
borrowing a framework in \cite{MR2305115} (see also \cite{MR2867756, MR4306765} for recent progress).
Known results for Hardy-type inequalities that seem to be particularly relevant to this study will be discussed in \S \ref{sec:setting}.
Throughout the paper, we assume only \eqref{eqn:coercive}-\eqref{eqn:homogenity}, \eqref{eqn:VD}-\eqref{eqn:PI}
and $(p, w)$-capacity density conditions.
The quantitative statements in results are new even for unweighted linear equations (compare with \cite[Remark 5.2]{MR856511}).
Our method seems to work for Cheeger differential equations on metric measure spaces
(see, \cite{MR1708448, MR1869615}, \cite[Chapter B.2]{MR2867756}, \cite{MR3842214}).
On the other hand, there are large gaps in its direct application to the minimizers of variational problems.
Another problem is that assumption in the results is clearly not optimal for interior regularity of solutions.
These problems are left for future work.

\subsection*{Organization of the paper}
In \S \ref{sec:setting}, we confirm of our problem and pick up related results on weighted Sobolev spaces.
In \S \ref{sec:op}, we define weak solutions to \eqref{eqn:p-laplace} and prove a Kato-type inequality.
In \S \ref{sec:reg}, we states regularity estimates that will be used in \S \ref{sec:barriers}.
In \S \ref{sec:barriers}, we construct strong barriers for \eqref{eqn:p-laplace}.
In \S \ref{sec:Dirichlet}, we apply the result in \S \ref{sec:barriers} to \eqref{eqn:p-laplace} and achieve our goals.

\subsection*{Notation}
Let $\Omega \subsetneq \R^{n}$ be an open set.
\begin{itemize}
\item
$\mathbf{1}_{E}(x) :=$ the indicator function of a set $E$.
\item
$C_{c}^{\infty}(\Omega) :=$
the set of all infinitely-differentiable functions with compact support in $\Omega$.
\item
$L^{p}(\Omega; \mu) :=$ the $L^{p}$ space with respect to a measure $\mu$ on $\Omega$.
\item
$f_{+} := \max\{f, 0\}$ and $f_{-} := - \min\{ f, 0 \}$.
\end{itemize}
For a closed set $\Gamma \subset \R^{n}$, we denote by $\delta_{\Gamma}$ the distance from $\Gamma$.
If $\Gamma = \R^{n} \setminus \Omega$, we write $\delta_{\R^{n} \setminus \Omega}$ as $\delta$ simply.
For a ball $B = B(x, R) = \{ y \colon \dist(x, y) < R \}$ and $\lambda > 0$,
we denote $B(x, \lambda R)$ by $\lambda B$.
The letters $c$ and $C$ denote various constants with and without indices.

\section{Setting and related work}\label{sec:setting}

\subsection{Admissible weights}
Let $1 < p < \infty$ be a fixed number.
A function $w \in L^{1}_{\loc}(\R^{n}; dx)$ ia called the \textit{weight} if $w(x) > 0$ a.e. in $\R^{n}$.
We write $w(E) = \int_{E} w \, dx$ for a Lebesgue measurable set $E \subset \R^{n}$.
Throughout the below, we assume that $w$ satisfying the doubling condition 
\begin{equation}\label{eqn:VD}
w(2B) \le C_{D} w(B)
\end{equation}
and the $p$-Poincar\'{e} inequality
\begin{equation}\label{eqn:PI}
\fint_{B} |u - u_{B}| \, d w \le C_{P} \, \diam(B) \left( \fint_{\lambda B} |\nabla u|^{p} \, d w \right)^{1 / p},
\quad
\forall u \in C_{c}^{\infty}(\R^{n}),
\end{equation}
where $B$ is an arbitrary ball in $\R^{n}$,
$\fint_{B} := w(B)^{-1} \int_{B}$, $u_{B} := \fint_{B} u \, d w$
and $C_{D}$, $C_{P}$ and $\lambda \ge 1$ are constants.
A weight $w$ satisfying \eqref{eqn:VD} and \eqref{eqn:PI} said to be \textit{$p$-admissible}.

It is well-known that \eqref{eqn:VD} and \eqref{eqn:PI} yield the following Sobolev type inequality:
\begin{equation}\label{eqn:sobolev}
\left( \fint_{B} |u|^{\chi p} \, d w \right)^{1 / \chi p}
\le
C \diam(B)
\left( \fint_{B} |\nabla u|^{p} \, d w \right)^{1 / p},
\quad \forall u \in C_{c}^{\infty}(B).
\end{equation}
where $C$ and $\chi > 1$ are constants depending only on $p$, $C_{D}$, $C_{P}$ and $\lambda$.
For detail, we refer to \cite[Chapter 20]{MR2305115} and the references cited therein.

Muckenhoupt $A_{p}$-weights are one typical example of $p$-admissible weights (\cite[Chapter 15]{MR2305115}).
The power function $|x|^{\mu}$ is an $A_{p}$-weight on $\R^{n}$ if and only if $-n < \mu < n(p - 1)$.
It seems to be known conventionally that if $\Omega$ is a bounded Lipschitz domain,
then $w(x) = \delta(x)^{\mu}$ is a $p$-admissible weight on $\R^{n}$ for $-1 < \mu < p - 1$. 
Finer results can be found in \cite{MR1021144, MR1118940, MR2606245, MR3900847} and \cite[Chapter 10]{MR4306765}.
Roughly speaking, if $\Gamma$ is an $s$-dimensional set with $0 < s < n$,
then $w(x) = \delta_{\Gamma}(x)^{\mu}$ is an $A_{p}$-weight on $\R^{n}$ for $-(n - s) < \mu < (n - s)(p - 1)$.

\subsection{Sobolev spaces and capacities}

The weighted Sobolev space $H^{1, p}(\Omega; w)$ is the closure of $C^{\infty}(\Omega)$
with respect to the norm
\[
\| u \|_{H^{1, p}(\Omega; w)}
:=
\left(
\int_{\Omega} |u|^{p} + |\nabla u|^{p} \, d w
\right)^{1 / p},
\]
where $\nabla u$ is the gradient of $u$ in the sense of $\R^{n}$. 
The corresponding local space $H^{1, p}_{\loc}(\Omega; w)$ is defined in the usual manner.
We denote by $H_{0}^{1, p}(\Omega; w)$ the closure of $C_{c}^{\infty}(\Omega)$ in $H^{1, p}(\Omega; w)$.
It is well-known that if $u, v \in H^{1, p}_{\loc}(\Omega; w)$, then $\min\{ u, v \} \in H^{1, p}_{\loc}(\Omega; w)$.

Let $O \subset \R^{n}$ be open, and let $K \subset O$ be compact.
The \textit{$(p, w)$-capacity} $\capacity_{p, w}(K, O)$
of the \textit{condenser} $(K, O)$ is defined by
\begin{equation}
\capacity_{p, w}(K, O)
:=
\inf \left\{
\| \nabla u \|_{L^{p}(O; w)}^{p} \colon u \geq 1 \ \text{on} \ K, \ u \in C_{c}^{\infty}(O)
\right\}.
\end{equation}

For a boundary point $\xi \in \partial \Omega$ (more generally, for $\xi \in \R^{n} \setminus \Omega$),
we consider the following \textit{$(p, w)$-capacity density condition}:
There exists $\gamma > 0$ such that
\begin{equation}\label{eqn:CDC}
\frac{ \capacity_{p, w}( \cl{B(\xi, R)} \setminus \Omega, B(\xi, 2R)) }{ \capacity_{p, w}( \cl{B(\xi, R)}, B(\xi, 2R)) } \ge \gamma.
\end{equation}
If \eqref{eqn:CDC} holds for all small $R > 0$, then $\xi$ is a regular point of the corresponding Dirichlet problem
(see Lemma \ref{lem:boundary_hoelder_esti} below, \cite[Theorem 6.31]{MR2305115} and the references cited therein).
Sufficient conditions for \eqref{eqn:CDC} via exterior corkscrew-type conditions
can be found in \cite[Theorem 6.31]{MR2305115} and \cite[Corollary 11.25]{MR2867756}.
In particular, every boundary point of a ball satisfies \eqref{eqn:CDC} for all $R > 0$.

When \eqref{eqn:CDC} holds for all $\xi \in \R^{n} \setminus \Omega$ and $R > 0$,
the set $\R^{n} \setminus \Omega$ is called as \textit{uniformly $(p, w)$-fat}.
This property is closely relates to Hardy-type inequalities.
The study of Hardy-type inequalities under capacity density conditions started with Ancona's work \cite{MR612540, MR856511}.
Lewis \cite{MR946438} developed theory of self-improvement property of uniformly $(p, 1)$-fat sets
and proved $p$-Hardy inequalities as a consequence.
Wannebo \cite{MR1010807} proved higher order Hardy-type inequalities using another approach.
Mikkonen \cite{MR1386213} studied nonlinear potential theory with $p$-admissible weights
and proved the self-improvement property of $(p, w)$-fat sets using a new method.
Bj\"{o}rn, MacManus and Shanmugalingam \cite{MR1869615} extended the results to the metric measure space setting.
The equivalence between pointwise Hardy inequalities and uniformly fatness was established by
several authors' contributions (see, \cite{MR1458875, MR1470421, MR2383525, MR2854110}).
We refer to \cite[Chapter 6]{MR4306765} and \cite[Chapter 3]{MR3408787} for further related work.
However, Ancona's proof in \cite{MR856511} is not very similar to any of the others.

The study of cases where \eqref{eqn:CDC} holds on a part of the boundary is more limited. 
The most similar known result to Corollary \ref{cor:hardy} below is \cite[Proposition 5.4]{MR3673660}.
They are using Wannebo's idea. 
In \cite{MR3361789}, another form of Hardy inequality and its applications are discussed
under assumptions on Sobolev extensibility of domains.

\section{Quasilinear elliptic operator}\label{sec:op}

Assume that $\mathcal{A} \colon \Omega \times \R^{n} \to \R^{n}$ satisfies \eqref{eqn:coercive}-\eqref{eqn:homogenity}.
For simplicity, we denote the following extended function $\overline{\mathcal{A}}$ by $\mathcal{A}$ again:
\[
\overline{  \mathcal{A} }(x, z)
=
\begin{cases}
\mathcal{A}(x, z) & x \in \Omega, \\ 
w(x) |z|^{p - 2} z & \text{otherwise},
\end{cases}
\quad
z \in \R^{n}.
\]
Let $f \in L^{1}_{\loc}(\Omega)$.
A function $u \in H^{1, p}_{\loc}(\Omega; w)$ is called weak (super-, sub-)solution to $- \divergence \mathcal{A}(x, \nabla u) = f$ in $\Omega$ if
\begin{equation}\label{eqn:weak-form}
\int_{\Omega} \mathcal{A}(x , \nabla u) \cdot \nabla \varphi \, dx = (\ge, \, \le) \int_{\Omega} \varphi f \,d x
\end{equation}
for any nonnegative $\varphi \in C_{c}^{\infty}(\Omega)$.

If $u$ is a weak supersolution to $- \divergence \mathcal{A}(x, \nabla u) = 0$ in $\Omega$,
then, by the Riesz representation theorem, there is a unique Radon measure $\nu[u]$ in $\Omega$ such that
\[
\int_{\Omega} \mathcal{A}(x , \nabla u) \cdot \nabla \varphi \, dx = \int_{\Omega} \varphi \, d \nu[u]
\]
for any $\varphi \in C_{c}^{\infty}(\Omega)$.
The measure $\nu[u]$ is called the \textit{Riesz measure} of $u$.
We refer to \cite[Chapter 21]{MR2305115} for further detail.

By \eqref{eqn:coercive}, if $u$ is a solution to $- \divergence \mathcal{A}(x, \nabla u) = f$ in $\Omega$,
then its truncation $\min\{ u, k \}$ is a supersolution to $- \divergence \mathcal{A}(x, \nabla u) = \min\{f, 0\}$ in $\Omega$.
The monotonicity condition \eqref{eqn:monotonicity} yields the following more general result.

\begin{lemma}\label{lem:barayage}\label{lem:glueing}
Assume that $u$ and $v$ are weak supersolutions to $- \divergence \mathcal{A}(x, \nabla u) = 0$ in $\Omega$.
Then,
\begin{equation}\label{eqn:glueing}
\int_{\Omega} \varphi \, d \nu[ \min\{ u, v \} ]
\ge
\int_{\Omega} \varphi \mathbf{1}_{ \{ u \le v \} } \, d \nu[ u ]
+
\int_{\Omega} \varphi \mathbf{1}_{ \{ u > v \} } \, d \nu[ v ]
\end{equation}
for any nonnegative $\varphi \in C_{c}^{\infty}(\Omega)$.
Moreover, if $u_{1}, \cdots, u_{k}$ are weak supersolutions to $- \divergence \mathcal{A}(x, \nabla u) = 0$ in $\Omega$,
and if there is a Radon measure $\sigma$ such that $\nu[u_{k}] \ll \sigma$ for all $k$, then
\begin{equation}\label{eqn:glueing2}
\nu[ \min\{ u_{1}, \cdots, u_{k} \} ] \ge \min\{ f_{1}, \cdots, f_{k} \} \sigma \quad \text{in} \ \Omega
\end{equation}
in the sense of distributions, where $f_{k}$ is the Radon-Nikod\'{y}m derivative of $\nu[u_{k}]$ with respect to $\sigma$.
\end{lemma}

\begin{proof}
Note that $u, v \in H^{1, p}_{\loc}(\Omega; w)$.
For $\epsilon > 0$, consider the functions
\[
\Phi_{\epsilon}(u - v) := \frac{\epsilon}{ (u - v)_{+} + \epsilon }
\quad \text{and} \quad
\Psi_{\epsilon}(u - v) := 1 - \Phi(u - v).
\]
These functions are globally Lipschitz continuous with respect to $u - v$;
therefore, $\Phi_{\epsilon}(u - v), \Psi_{\epsilon}(u - v) \in H^{1, p}_{\loc}(\Omega; w)$.
We also note that $\Phi_{\epsilon}(u - v)(x) \to \mathbf{1}_{ \{ u \le v \} }(x)$ for all $x$.
Fix $\varphi \in C_{c}^{\infty}(\Omega)$. Then,
\[
\begin{split}
&
\int_{\Omega} \varphi \, d \nu[ \min\{ u, v \} ] \\
& \quad
=
\int_{\Omega} \varphi \Phi_{\epsilon}(u - v) \, d \nu[ \min\{ u, v \} ]
+
\int_{\Omega} \varphi \Psi_{\epsilon}(u - v) \, d \nu[ \min\{ u, v \} ].
\end{split}
\]
By the definition of $\nu[ \min\{ u, v \} ]$, we have
\[
\begin{split}
\int_{\Omega} \varphi \Phi_{\epsilon}(u - v) \, d \nu[ \min\{ u, v \} ]
& =
\int_{ [u \le v] } \mathcal{A}(x, \nabla u) \cdot \nabla ( \varphi \Phi_{\epsilon}(u - v) ) \, dx \\
& \quad +
\int_{ [u > v] } \mathcal{A}(x, \nabla v) \cdot \nabla ( \varphi \Phi_{\epsilon}(u - v) ) \, dx.
\end{split}
\]
Note that
$\left( \mathcal{A}(x, \nabla v) - \mathcal{A}(x, \nabla u) \right) \cdot \nabla \Phi_{\epsilon}(u - v) \ge 0$ by \eqref{eqn:monotonicity}.
Therefore,
\[
\int_{ [u > v] } \mathcal{A}(x, \nabla v) \cdot \nabla ( \varphi \Phi_{\epsilon}(u - v) ) \, dx + I_{\epsilon}
\ge
\int_{ [u > v] } \mathcal{A}(x, \nabla u) \cdot \nabla ( \varphi \Phi_{\epsilon}(u - v) ) \, dx,
\]
where
\[
I_{\epsilon}
:=
\int_{ [u > v] } \left( \mathcal{A}(x, \nabla u) - \mathcal{A}(x, \nabla v) \right) \cdot \nabla \varphi \, \Phi_{\epsilon}(u - v) \, dx.
\]
Adding the two inequalities, we get
\begin{equation}\label{eqn:glue_01}
\int_{\Omega} \varphi \Phi_{\epsilon}(u - v) \, d \nu[ \min\{ u, v \} ] + I_{\epsilon}
\ge
\int_{\Omega} \varphi \Phi_{\epsilon}(u - v) \, d \nu[ u ].
\end{equation}
Similarly, since
\[
\begin{split}
\int_{\Omega} \varphi \Psi_{\epsilon}(u - v) \, d \nu[ \min\{ u, v \} ]
& =
\int_{ [u \le v] } \mathcal{A}(x, \nabla u) \cdot \nabla ( \varphi \Psi_{\epsilon}(u - v) ) \, dx \\
& \quad +
\int_{ [u > v] } \mathcal{A}(x, \nabla v) \cdot \nabla ( \varphi \Psi_{\epsilon}(u - v) ) \, dx
\end{split}
\]
and
$\left( \mathcal{A}(x, \nabla u) - \mathcal{A}(x, \nabla v) \right) \cdot \nabla \Psi_{\epsilon}(u - v) \ge 0$,
we have
\begin{equation}\label{eqn:glue_02}
\int_{\Omega} \varphi \Psi_{\epsilon}(u - v) \, d \nu[ \min\{ u, v \} ]
+
II_{\epsilon}
\ge
\int_{\Omega} \varphi \Psi_{\epsilon}(u - v) \, d \nu[ v ],
\end{equation}
where
\[
II_{\epsilon}
:=
\int_{ [u \le v] } \left( \mathcal{A}(x, \nabla v) - \mathcal{A}(x, \nabla u) \right) \cdot \nabla \varphi \, \Psi_{\epsilon}(u - v) \, dx.
\]
Combining \eqref{eqn:glue_01} and \eqref{eqn:glue_02}, we obtain
\[
\int_{\Omega} \varphi \, d \nu[ \min\{ u, v \} ]
+
I_{\epsilon} + II_{\epsilon}
\ge
\int_{\Omega} \varphi \Phi_{\epsilon}(u - v) \, d \nu[ u ]
+
\int_{\Omega} \varphi \Psi_{\epsilon}(u - v) \, d \nu[ v ].
\]
Take the limit $\epsilon \to 0$.
By the dominated convergence theorem, the right-hand side of this inequality goes to the right-hand of \eqref{eqn:glueing}.
By the same reason, $I_{\epsilon}, II_{\epsilon} \to 0$.
Therefore, the former statement holds.
The latter statement is a consequence of induction.
\end{proof}

\section{Regularity estimates}\label{sec:reg}

By standard techniques in the De Giorgi-Nash-Moser theory,
the Sobolev inequality \eqref{eqn:sobolev} yields the following global $L^{\infty}$ estimate and weak Harnack inequality.
The proofs are combinations of \cite[p.63, Lemma B.2]{MR567696}, \cite[Theorem 8.18]{MR1814364}
and \cite[Theorems 3.59 and 7.46]{MR2305115}. See also \cite[Chapter 3.1.0]{MR757718}.

\begin{lemma}\label{lem:global_boundedness}
Let $u \in H^{1, p}(\Omega; w)$ be a weak subsolution to
$- \divergence \mathcal{A}(x, \nabla u) = f$ in $\Omega$.
Let $F =  \left( \diam( \Omega )^{p} \| f_{+} / w\|_{L^{\infty}(\Omega)} \right)^{1 / (p - 1)}$.
Then,
\[
\esssup_{\Omega} u
\le
\sup_{\partial \Omega} u + C F_{+},
\]
where $\sup_{\partial \Omega} u := \inf \{ k \in \R^{n} \colon (u - k)_{+} \in H_{0}^{1, p}(\Omega; w) \}$
and $C$ is a constant depending only on $p$, $C_{D}$ and $\{ C_{P}, \lambda \}$.
In particular, there is a constant $c_{1} = c_{1}(p, C_{D}, \{ C_{P}, \lambda \})$ such that
if $\| f_{+} / w \|_{L^{\infty}(\Omega)} \le c_{1} \diam(\Omega)^{p}$, then
$\esssup_{\Omega} u \le \sup_{\partial \Omega} u + 1 / 4$.
\end{lemma}

\begin{lemma}\label{lem:weak_harnack}
Let $u \in H^{1, p}(2B; w)$ be a nonnegative weak supersolution to
$- \divergence \mathcal{A}(x, \nabla u) = f$ in $2B$,
and let $F_{-} =  \left( R^{p} \| f_{-} / w\|_{L^{\infty}(2B)} \right)^{1 / (p - 1)}$.
Then,
(i) for each $0 < s < \chi (p - 1)$, there exists a constant $C$ depending only on $p$, $C_{D}$, $\{ C_{P}, \lambda \}$, $L$ and $s$ such that
\begin{equation}\label{eqn:WHI}
\left( \fint_{B} u^{s} \, d w \right)^{1 / s}
\le
C \left( \essinf_{B} u + F_{-} \right).
\end{equation}
(ii) there exists a constant $C$ depending only on $p$, $C_{D}$, $\{C_{P}, \lambda \}$ and $L$ such that
\begin{equation}\label{eqn:WHI2}
R^{p - 1} \fint_{B} |\nabla u|^{p - 1} \, d w
\le
C \left( \essinf_{B} u + F_{-} \right)^{p - 1}.
\end{equation}
\end{lemma}

It is well-known that Lemma \ref{lem:weak_harnack} yields a local H\"{o}lder estimate (see, e.g., \cite[Theorem 8.24]{MR1814364}).
Below, we always assume that $f / w$ is locally bounded.
Under this assumption, any weak solution $u$ to $- \divergence \mathcal{A}(x, \nabla u) = f$
can be regard as a locally H\"{o}lder continuous function.

We use the exterior condition \eqref{eqn:CDC} via the following boundary H\"{o}lder estimate.
The proof below is the same as \cite[pp. 206-209]{MR1814364} (see also \cite{MR492836} and \cite{MR757718}).

\begin{lemma}\label{lem:boundary_hoelder_esti}
Let $B$ be a ball centered at $\xi \in \partial \Omega$ with radius $R$. 
Assume that \eqref{eqn:CDC} holds.
Let $u \in H^{1, p}(\Omega; w) \cap L^{\infty}(\Omega)$
satisfy $- \divergence \mathcal{A}(x, \nabla u) = f$ in $\Omega \cap 4B$.
Let $F_{\pm} =  \left( R^{p} \| f_{\pm} / w\|_{L^{\infty}(2B)} \right)^{1 / (p - 1)}$.
Then,
\[
\osc_{\Omega \cap B} u
\le
\left(1 - \frac{\gamma^{1 / (p - 1)}}{C} \right) \osc_{\Omega \cap 4B} u
+
\frac{\gamma^{1 / (p - 1)}}{C} \osc_{\partial \Omega \cap B} u
+
F_{+} + F_{-},
\]
where $C$ is a constant depending only on $p$, $C_{D}$, $\{ C_{P}, \lambda \}$ and $L$ and $\osc := \sup - \inf$.
\end{lemma}

\begin{proof}
We first assume that $u$ is a nonnegative supersolution to $- \divergence \mathcal{A}(x, \nabla u) = f$ in $\Omega \cap 4B$.
Let $m = \inf_{\partial \Omega \cap B} u$, and let
\[
u_{m}^{-}(x)
=
\begin{cases}
\min\{ u(x), m \} & x \in \Omega, \\
m & \text{otherwise}.
\end{cases}
\]
Note that $0 \le u_{m}^{-} \le m$ and $u_{m}^{-} = m$ on $\cl{B} \setminus \Omega$.
Take $\eta \in C_{c}^{\infty}(2B)$ such that $0 \le \eta \le 1$ in $2B$, $\eta$ = 1 on $B$ and $|\nabla \eta| \le C / R$.
Then, \eqref{eqn:CDC} gives
\[
m^{p} \gamma
\le
\frac{ m^{p} \capacity_{p, w}( \cl{B} \setminus \Omega, 2B) }{ \capacity_{p, w}( \cl{B}, 2B) }
\le
C R^{p} \fint_{2B} |\nabla (u_{m}^{-} \eta)|^{p} \, d w.
\]
By the product rule, $\nabla (u_{m}^{-} \eta) = \nabla u_{m}^{-} \eta + u_{m}^{-} \nabla \eta$ a.e.
Therefore,
\begin{equation}\label{eqn:bdry_esti_01}
\int_{2B} |\nabla (u_{m}^{-} \eta)|^{p} \, d w
\le
C \left( R^{-p} \int_{2B} (u_{m}^{-})^{p} d \mu + \int_{2B} |\nabla u_{m}^{-}|^{p} \eta^{p} \, d w \right).
\end{equation}
By Lemma \ref{lem:weak_harnack} (i), the former term on the right-hand side is estimated by
\begin{equation}\label{eqn:bdry_esti_02}
\fint_{2B} (u_{m}^{-})^{p} \, d w
\le
m \fint_{2B} (u_{m}^{-})^{p - 1} \, d w
\le
C m \left( \inf_{B} u_{m}^{-} + F_{-} \right)^{p - 1}.
\end{equation}
Consider the test function $(m - u_{m}^{-}) \eta^{p}$.
Since $u_{m}^{-}$ is a supersolution to $- \divergence \mathcal{A}(x, \nabla u) = \min\{f, 0\}$ in $\Omega$, we have
\[
\begin{split}
&
\int_{\Omega} \mathcal{A}(x, \nabla u_{m}^{-}) \cdot \nabla (m - u_{m}^{-}) \eta^{p} \, dx
+
p \int_{\Omega} \mathcal{A}(x, \nabla u_{m}^{-}) \cdot \nabla \eta \eta^{p - 1} (m - u_{m}^{-}) \, dx
\\
& \quad
\ge
\int_{\Omega} (m - u_{m}^{-}) \eta^{p} \frac{\min\{ f, 0 \}}{w} \, d w.
\end{split}
\]
By \eqref{eqn:coercive} and \eqref{eqn:growth}, this inequality yields
\[
\begin{split}
\int_{2B} |\nabla u_{m}^{-}|^{p} \eta^{p} \, d w
& =
\int_{2B} |\nabla (m - u_{m}^{-})|^{p} \eta^{p} \, d w
\\
& \le
C m 
\left(
R^{-1} \int_{2B} |\nabla u_{m}^{-}|^{p - 1} \, d w
+
F_{-}^{p - 1} R^{-p} w(2B)
\right).
\end{split}
\]
By Lemma \ref{lem:weak_harnack} (ii), the former term on the right-hand side is estimated by
\begin{equation}\label{eqn:bdry_esti_03}
R^{-1} \fint_{2B} |\nabla u_{m}^{-}|^{p - 1} \, d w
\le
C R^{-p} \left( \inf_{B} u_{m}^{-} + F_{-} \right)^{p - 1}.
\end{equation}
Combining \eqref{eqn:bdry_esti_01}, \eqref{eqn:bdry_esti_02} and \eqref{eqn:bdry_esti_03}, we obtain
\begin{equation}\label{eqn:boundary_esti}
m
\le
\frac{C}{\gamma^{1 / (p - 1)}}
\left(
\inf_{B} u_{m}^{-} + F_{-}
\right).
\end{equation}

Let $M(R) = \sup_{\Omega \cap B(\xi, R)} u$ and $m(R) = \inf_{\Omega \cap B(\xi, R)} u$.
Applying \eqref{eqn:boundary_esti} to $M(4R) - u$ and $u - m(4R)$, we obtain
\begin{equation}\label{eqn:bdry_esti_04}
M(4R) - \sup_{\partial \Omega \cap B} u 
\le
\frac{C}{\gamma^{1 / (p - 1)}} \left( M(4R) - M(R) + F_{+} \right),
\end{equation}
\[
\inf_{\partial \Omega \cap B} u - m(4R)
\le
\frac{C}{\gamma^{1 / (p - 1)}} \left( m(R) - m(4R) + F_{-} \right).
\]
Adding the two inequalities, we arrive at the desired estimate.
\end{proof}

Using \eqref{eqn:bdry_esti_04} in Lemma \ref{lem:boundary_hoelder_esti} iteratively, we get the following lemma.

\begin{lemma}\label{lem:boundary_hoelder_esti2}
Let $B$ be a ball centered at $\xi \in \partial \Omega$ with radius $R_{0}$. 
Assume that \eqref{eqn:CDC} holds for all $0 < R \le R_{0}$.
Then, there are positive constants $c_{2}$ and $\theta \in (0, 1)$  depending only on $p$, $C_{D}$, $\{C_{P}, \lambda \}$, $L$ and $\gamma$
such that if $u \in H_{0}^{1, p}(\Omega; w)$ is a nonnegative bounded weak subsolution to
$- \divergence \mathcal{A}(x, \nabla u) = f$ in $\Omega \cap B$, and
if $\| f_{+} / w \|_{L^{\infty}(\Omega)} \le c_{2} R_{0}^{-p}  \left( \sup_{\Omega \cap B} u \right)^{p - 1}$, then
\[
\sup_{\Omega \cap \theta B} u
\le
\frac{1}{4} \sup_{\Omega \cap B} u.
\]
\end{lemma}

\section{Construction of strong barriers}\label{sec:barriers}

\begin{theorem}\label{thm:main}
Let $\Gamma \subset \partial \Omega$ be a closet set, and assume that \eqref{eqn:CDC} holds for all $\xi \in \Gamma$ and $R > 0$.
Then, there exists a nonnegative function $s_{\Gamma} \in H^{1, p}_{\loc}(\Omega; w) \cap C(\Omega)$ satisfying
\begin{equation}\label{eqn:barrier}
- \divergence \mathcal{A}(x, \nabla s_{\Gamma}) \ge c_{H} \frac{s_{\Gamma}(x)^{p - 1}}{\delta_{\Gamma}(x)^{p}} w(x) \quad \text{in} \ \Omega
\end{equation}
and
\begin{equation}\label{eqn:bound_of_barrier}
\delta_{\Gamma}(x)^{\alpha} \le s_{\Gamma}(x) \le 30 \, \delta_{\Gamma}(x)^{\alpha}
\end{equation}
for all $x \in \Omega$,
where, $c_{H}$ and $\alpha > 0$ are constants depending only on $p$, $C_{D}$, $\{ C_{P}, \lambda \}$, $L$ and $\gamma$.
\end{theorem}

\begin{lemma}\label{lem:auxiliary_func}
Let $B$ be a ball centered at $\xi \in \partial \Omega$ with radius $R_{0}$. 
Assume that \eqref{eqn:CDC} holds for all $0 < R \le R_{0}$.
Let $c_{3} := \min\{ c_{1}, c_{2} \}$, where $c_{1}$ and $c_{2}$ are constants in
Lemmas \ref{lem:global_boundedness} and \ref{lem:boundary_hoelder_esti2}.
Then, there exists a function $u_{B} \in H^{1, p}_{\loc}(\Omega; w) \cap C(\Omega)$ satisfying
\begin{equation}\label{eqn:auxiliary_func}
- \divergence \mathcal{A}(x, \nabla u_{B}) = c_{3} R_{0}^{-p} w \quad \text{in} \ \Omega \cap B,
\end{equation}
\begin{equation}\label{eqn:auxiliary_func_01}
 \frac{1}{4} \le u_{B} \le \frac{5}{4} \quad \text{in} \ \Omega,
\end{equation}
\begin{equation}\label{eqn:auxiliary_func_02}
u_{B} = 1 \quad \text{on} \ \Omega \setminus B,
\quad
u_{B} \le \frac{1}{2} \quad \text{on} \ \Omega \cap \theta B,
\end{equation}
where $\theta$ is a constant in Lemma \ref{lem:boundary_hoelder_esti2}.
\end{lemma}

\begin{proof}
Existence of the function follows from the standard argument in \cite[p.177]{MR0259693}.
Take $\eta_{B} \in C^{\infty}(\cl{\Omega})$ such that
\[
\eta_{B} = \frac{1}{4} \ \text{on} \ \cl{\Omega \cap B / 2},
\quad
\eta_{B} = 1 \ \text{on} \ \cl{\Omega} \setminus B,
\quad
\frac{1}{4} \le \eta_{B} \le 1 \ \text{in} \ \cl{\Omega}.
\]
Consider the Dirichlet problem
\[
\begin{cases}
\displaystyle
- \divergence A_{B}(x, \nabla v_{B}) = c_{3} R_{0}^{-p} w & \text{in} \ \Omega \cap B,
\\
v_{B} \in H_{0}^{1, p}(\Omega \cap B; w),
\end{cases}
\]
where
\[
A_{B}(x, z) := \mathcal{A}(x, z + \nabla \eta_{B}(x)).
\]
By \eqref{eqn:sobolev}, the right-hand side belongs to the dual of $H_{0}^{1, p}(\Omega \cap B; w)$.
Therefore, by the Minty-Browder theorem, 
there exists a unique solution $u_{B} \in \eta_{B} + H_{0}^{1, p}(\Omega \cap B; w)$ to \eqref{eqn:auxiliary_func}.
By Lemma \ref{lem:weak_harnack}, $u_{B}$ is continuous in $\Omega \cap B$.
Meanwhile, if $x \in \Omega \cap \partial B$, then  $u_{B}$ is continuous at $x$ by Lemma \ref{lem:boundary_hoelder_esti}.
Consequently, $u_{B}$ is continuous in $\Omega$.
By the comparison principle, $u_{B} \ge 1 / 4$ in $\Omega$.
By Lemma \ref{lem:global_boundedness}, $u_{B} \le \sup_{\partial \Omega} u + 1 / 4 \le 5 / 4$ in $\Omega$.
The latter bound in \eqref{eqn:auxiliary_func_02} follows from Lemma \ref{lem:boundary_hoelder_esti2}.
\end{proof}

\begin{proof}[Proof of Theorem \ref{thm:main}]
First, we construct a function $s$ by the following two steps.
(i) For $k \in \Z$, set $R_{k} = (\theta / 2)^{k}$ and $E_{k} = \{ x \in \Omega \colon \delta_{\Gamma}(x) \le R_{k} \}$.
Choosing  $\{ \xi_{j} \}_{j \in J_{k}} \subset \Gamma$, we construct a locally finite covering $\{ B(\xi_{j}, \theta R_{k} ) \}_{j \in J_{k}}$ of $E_{k + 1}$.
Note that
\[
E_{k + 1}
\subset
D_{k}
:=
\Omega \cap \bigcup_{j \in J_{k}} B(\xi_{j}, R_{k})
\subset
E_{k}.
\]
Using Lemma \ref{lem:auxiliary_func},
we define a function $v_{k} \in H^{1, p}_{\loc}(\Omega; w) \cap C(\Omega)$ by
\[
v_{k} = \inf_{j \in J_{k}} u_{B(\xi_{j}, R_{k})}.
\]
By \eqref{eqn:auxiliary_func_01} and \eqref{eqn:auxiliary_func_02}, we have
\begin{equation}\label{eqn:bound_of_vk}
\frac{1}{4} \le v_{k} \le \frac{5}{4} \ \text{in} \ \Omega
\end{equation}
and
\begin{equation}\label{eqn:overlap}
v_{k} = 1 \ \text{on} \ \Omega \setminus D_{k},
\quad
v_{k} \le \frac{1}{2} \ \text{on} \  E_{k + 1}.
\end{equation}
Moreover, by Lemma \ref{lem:glueing}, we have
\begin{equation}\label{eqn:rigidity_eqn}
- \divergence \mathcal{A}(x, \nabla v_{k}) 
\ge
c_{3} R_{k}^{-p} w
\quad \text{in} \ D_{k}.
\end{equation}
(ii) Define a function $s$ on $\Omega$ by
\[
s(x) = \inf_{k \in \Z} \left\{ \left( \frac{3}{4} \right)^{k} \tilde{v}_{k}(x) \right\},
\quad \text{where} \quad
\tilde{v}_{k}(x) 
=
\begin{cases}
v_{k}(x) & \text{if} \ x \in E_{k}, \\
+ \infty & \text{otherwise}.
\end{cases}
\]
By \eqref{eqn:bound_of_vk} and the inequality $\left( 3 / 4 \right)^{6} \le 1 / 5 < \left( 3 / 4 \right)^{5}$,
\[
\left( \frac{3}{4} \right)^{k - 6} v_{k - 6}(x)
\ge
\left( \frac{3}{4} \right)^{k - 6} \frac{1}{4}
\ge
\left( \frac{3}{4} \right)^{k} \frac{5}{4}
\ge
\left( \frac{3}{4} \right)^{k} v_{k}(x)
\]
for any $x \in E_{k}$.
Therefore,
\begin{equation}\label{eqn:bound_of_s}
s(x) = \min\left\{ \left( \frac{3}{4} \right)^{k - 5} v_{k - 5}(x), \cdots, \left( \frac{3}{4} \right)^{k} v_{k}(x) \right\}
\end{equation}
for all $x \in E_{k} \setminus E_{k + 1}$.
In particular, $s \in H^{1, p}_{\loc}(\Omega; w) \cap C(\Omega)$.

Next, we claim that
\begin{equation}\label{eqn:Riesz_mass}
- \divergence \mathcal{A}(x, \nabla s)
\ge
\left( \frac{3}{4} \right)^{k (p - 1)} c_{3} R_{k - 5}^{-p} w
\end{equation}
in an open neighborhood of $E_{k} \setminus E_{k + 1}$.
By Lemma \ref{lem:glueing} and \eqref{eqn:rigidity_eqn}, this inequality holds in $D_{k} \setminus E_{k + 1}$.
Meanwhile, by \eqref{eqn:overlap}, we have
\[
v_{k - 1} \le \frac{1}{2} < \frac{3}{4} v_{k} \quad \text{in} \ E_{k} \setminus D_{k}.
\]
By continuity of $v_{k}$ and $v_{k - 1}$,
we can take an open set $O_{k}$ such that $v_{k - 1} < (3 / 4) v_{k}$ in $O_{k}$.
Note that
\begin{equation*}
s(x) = \min\left\{ \left( \frac{3}{4} \right)^{k - 5} v_{k - 5}(x), \cdots, \left( \frac{3}{4} \right)^{k - 1} v_{k - 1}(x) \right\}
\end{equation*}
for all $x \in O_{k}$.
Since $v_{k - 5}, \dots, v_{k - 1}$ satisfy \eqref{eqn:rigidity_eqn} in $D_{k - 1}$,
\eqref{eqn:Riesz_mass} holds in $O_{k} \cap D_{k  - 1}$ by the same reason as above.
The open set $(D_{k} \setminus E_{k + 1}) \cup (O_{k} \cap D_{k - 1})$ has the desired property.

Finally, we consider pointwise behavior of $s$.
It follows from \eqref{eqn:bound_of_s} that
\begin{equation}\label{eqn:bound_of_s2}
\frac{1}{4} \left( \frac{3}{4} \right)^{k} \le s(x) \le \frac{5}{4} \left( \frac{4}{3} \right)^{5} \left( \frac{3}{4} \right)^{k}
\end{equation}
for all $x \in E_{k} \setminus E_{k + 1}$.
By this inequality and the definition of $R_{k}$,
the right-hand side of \eqref{eqn:Riesz_mass} is estimated from below by
\[
\left( \frac{3}{4} \right)^{k (p - 1)} c_{3} R_{k - 5}^{-p}
\ge
s^{p - 1} \frac{c_{H}}{\delta_{\Gamma}(x)^{p}},
\]
where $c_{H} = \left\{ (5 / 4) \left( 4 / 3 \right)^{5} \right\}^{1 - p}  \left( \theta / 2 \right)^{6p} c_{3}$.
Hence, $s$ satisfies \eqref{eqn:barrier}.
Take $\alpha > 0$ such that $3 / 4 = (\theta / 2)^{\alpha}$.
Then, \eqref{eqn:bound_of_s2} gives
\[
\frac{1}{4} \delta_{\Gamma}(x)^{\alpha}
\le
s(x)
\le
\frac{5}{4} \left( \frac{4}{3} \right)^{6} \delta_{\Gamma}(x)^{\alpha}
\le
\frac{30}{4} \delta_{\Gamma}(x)^{\alpha}
\]
for all $x \in \Omega$.
Thus, the function $s_{\Gamma} := 4 s$ has the desired properties.
\end{proof}

\begin{corollary}\label{cor:hardy}
Assume that $w$ is a $p$-admissible weight on $\R^{n}$.
Let $\Gamma \subset \partial \Omega$ be a nonempty closed set,
and assume that \eqref{eqn:CDC} holds for all $\xi \in \Gamma$ and $R > 0$.
Let $c_{H} = c_{H}(p, C_{D}, \{ C_{P}, \lambda \}, L = 1, \gamma) > 0$ be the constant in Theorem \ref{thm:main}.
Then,
\[
c_{H} \int_{\Omega} \frac{|\varphi|^{p}}{ \delta_{\Gamma}^{p}} \, d w \le \int_{\Omega} |\nabla \varphi|^{p} \, d w
\]
for all $\varphi \in C_{c}^{\infty}(\Omega)$.
\end{corollary}

\begin{proof}
Applying Theorem \ref{thm:main} to $\mathcal{A}(x, z) = w(x) |z|^{p - 2} z$,
we get a nonnegative function $s_{\Gamma} \in H_{\loc}^{1, p}(\Omega; w) \cap C(\Omega)$ satisfying
\[
- \laplacian_{p, w} s_{\Gamma} \ge c_{H} \frac{s_{\Gamma}(x)^{p - 1}}{\delta_{\Gamma}(x)^{p}} w(x) \quad \text{in} \ \Omega.
\]
Take $\varphi \in C_{c}^{\infty}(\Omega)$ such that $\varphi \ge 0$. Then,
\[
c_{H} \int_{\Omega} \frac{\varphi^{p}}{\delta_{\Gamma}^{p}} \, d w
=
c_{H} \int_{\Omega} \frac{\varphi^{p}}{s_{\Gamma}^{p - 1}} \frac{s_{\Gamma}^{p - 1}}{\delta_{\Gamma}^{p}} \, d w
\le
\int_{\Omega} \varphi^{p} \frac{d \nu[s_{\Gamma}]}{ s_{\Gamma}^{p - 1} },
\]
where $\nu[s_{\Gamma}]$ is the Riesz measure of $s_{\Gamma}$ with respect the $(p, w)$-Laplacian.
Applying the Picone inequality (\cite[Theorem 1.1]{MR1618334}) to the right-hand side (for detail, see e.g. \cite[Lemma 3.2]{MR4079054}), we get
\[
\int_{\Omega} \varphi^{p} \frac{d \nu[s_{\Gamma}]}{s_{\Gamma}^{p - 1}}
\le
\int_{\Omega} |\nabla \varphi|^{p} \, d w.
\]
Combining the two inequalities, we obtain the desired inequality.
\end{proof}

\begin{remark}
If $\R^{n} \setminus \Omega$ is uniformly $(p, w)$-fat, then
\[
c_{H} \int_{\Omega} \frac{|\varphi|^{p}}{ \delta^{p}} \, d w \le \int_{\Omega} |\nabla \varphi|^{p} \, d w
\]
for all $\varphi \in C_{c}^{\infty}(\Omega)$.
This inequality was first proved in \cite[Theorem 8.15]{MR1386213} with an additional assumption on $w$.
Later, Keith and Zhong \cite{MR2415381} proved the self-improvement property of $p$-admissible weights, and the assumption was dropped.
We note that our proof is not via the self-improvement property of $(p, w)$-fat sets and does not use of this result.
\end{remark}

\section{Applications to Dirichlet problems}\label{sec:Dirichlet}


\begin{lemma}\label{lem:supersol}
Assume that there exists a function $s_{\Gamma}$ on $\Omega$ satisfying \eqref{eqn:barrier} and \eqref{eqn:bound_of_barrier}.
Let $h \colon (0, \infty) \to (0, \infty)$ be a continuously differentiable nondecreasing concave function such that
\[
g(s)
:=
\int_{0}^{s} h(t) \frac{dt}{t} < \infty
\]
for some $s > 0$.
Then,  $v(x) := g(s_{\Gamma}(x)) \in H^{1, p}_{\loc}(\Omega; w) \cap C(\Omega)$ is a nonnegative weak supersolution to
\[
- \divergence \mathcal{A}(x, \nabla v)
\ge
c_{H} \frac{h(\delta_{\Gamma}(x)^{\alpha})^{p - 1}}{\delta_{\Gamma}(x)^{p}} w(x) \quad \text{in} \ \Omega.
\]
Moreover, $g(\delta_{\Gamma}(x)^{\alpha}) \le v(x) \le g( 30 \, \delta_{\Gamma}(x)^{\alpha})$ for all $x \in \Omega$.
\end{lemma}

\begin{proof}
For simplicity, we denote $s_{\Gamma}$ by $s$.
Let $' = \frac{d}{dt}$.
By assumption, $h'(t) t \le h(t)$ for all $t \in (0, \infty)$; therefore, $g$ is increasing and concave on $(0, \infty)$.
Fix a nonnegative function $\varphi \in C_{c}^{\infty}(\Omega)$.
Since $g'' \le 0$, we have
\[
\int_{\Omega} \mathcal{A}(x, \nabla s) \cdot \nabla \left( g'(s)^{p - 1} \right) \varphi \, dx \le 0.
\]
By the chain rule, \eqref{eqn:homogenity} and \eqref{eqn:barrier}, we get
\[
\begin{split}
\int_{\Omega} \mathcal{A}(x, \nabla g(s)) \cdot \nabla \varphi \, dx
& =
\int_{\Omega} \mathcal{A}(x, g'(s) \nabla s) \cdot \nabla \varphi \, dx \\
& \ge
\int_{\Omega} \mathcal{A}(x, \nabla s) \cdot \nabla \left( g'(s)^{p - 1} \varphi \right) \, dx \\
& \ge
c_{H} \int_{\Omega} \frac{s^{p - 1}}{\delta_{\Gamma}^{p}} \left( g'(s)^{p - 1} \varphi \right) \, d w.
\end{split}
\]
Since $h$ is nondecreasing, \eqref{eqn:bound_of_barrier} gives
\[
s^{p - 1} g'(s)^{p - 1}
=
h\left( s \right)^{p - 1}
\ge
h( \delta_{\Gamma}^{\alpha} )^{p - 1}.
\]
The latter statement is a consequence of the monotonicity of $g$.
\end{proof}


\begin{proposition}\label{prop:Dirichlet}
Let $\Omega$ be a bounded open set,
and let $\Gamma \subset \partial \Omega$ be a nonempty closed set satisfying \eqref{eqn:CDC} for all $\xi \in \Gamma$ and $R > 0$.
Let $f \in L^{\infty}_{\loc}(\Omega)$, and assume that there exists a function $h$ satisfying the assumption in Lemma \ref{lem:supersol} and
\[
|f(x)| \le c_{H} \frac{h(\delta_{\Gamma}(x)^{\alpha})^{p - 1}}{\delta_{\Gamma}(x)^{p}} w(x)
\]
for almost every $x \in \Omega$, where $c_{H}$ and $\alpha$ are constants in Theorem \ref{thm:main}.
Then, there exists a nonnegative unique weak solution $u \in H_{\loc}^{1, p}(\Omega; w) \cap C(\Omega)$ to \eqref{eqn:p-laplace}
in the sense that (i) the zero extension of $u$ belongs to $H^{1, p}_{\loc}(\R^{n} \setminus \Gamma; w)$,
and (ii) $\lim_{x \to \xi} u(x) = 0$ for all $\xi \in \Gamma$.
\end{proposition}

\begin{proof}
Using Theorem \ref{thm:main} and Lemma \ref{lem:supersol},
we get a nonnegative supersolution $v(x) = g(s(x))$ to $- \divergence \mathcal{A}(x, \nabla u) = |f|$ in $\Omega$.
Set $D_{k} = \{ x \in \Omega \colon \dist(x, \Gamma) > 1 / k \}$,
and consider the sequence of weak solutions $\{ v_{k} \}_{k = 1}^{\infty} \subset H_{0}^{1, p}(\Omega; w) \cap C(\Omega)$ to
\[
\begin{cases}
- \divergence \mathcal{A}(x, \nabla v_{k}) = |f| \mathbf{1}_{D_{k}} & \text{in} \ \Omega, \\
v_{k} = 0 & \text{on} \ \partial \Omega.
\end{cases}
\]
Since $\Omega$ is bounded, the right-hand side belongs to the dual of $H_{0}^{1, p}(\Omega; w)$.
By the comparison principle in \cite[Theorem 3.5]{MR4493271},
\begin{equation}\label{eqn:comparison_with_v}
0 \le v_{k}(x)  \le v(x)
\end{equation}
for all $x \in \Omega$.
Also, consider  the sequence of weak solutions $\{ u_{k} \}_{k = 1}^{\infty} \subset H_{0}^{1, p}(\Omega; w) \cap C(\Omega)$ to
\[
\begin{cases}
- \divergence \mathcal{A}(x, \nabla u_{k}) = f \mathbf{1}_{D_{k}}  & \text{in} \ \Omega, \\
u_{k} = 0 & \text{on} \ \partial \Omega.
\end{cases}
\]
By the comparison principle for weak solutions, we have
\[
- v_{k}(x) \le u_{k}(x) \le v_{k}(x)
\]
for all $x \in \Omega$; therefore,
\begin{equation}\label{eqn:bound_by_v}
|u_{k}(x)| \le v(x)
\end{equation}
for all $x \in \Omega$.
Fix $j \ge 1$, and take a nonnegative function $\eta \in C^{\infty}(\R^{n})$ such that $\eta \equiv 1$ on $\cl{D_{j}}$
and $\eta \equiv 0$ on $\Omega \setminus D_{j + 1}$.
Testing the equation of $u_{k}$ with $u_{k} \eta^{p}$, we get
\[
\int_{\Omega} \mathcal{A}(x , \nabla u_{k}) \cdot \nabla u_{k} \eta^{p} \, dx
+
p \int_{\Omega} \mathcal{A}(x , \nabla u_{k}) \cdot \nabla \eta \eta^{p - 1} u_{k} \, dx
= 
\int_{\Omega} f u_{k} \eta^{p} \, dx.
\]
By \eqref{eqn:coercive} and \eqref{eqn:growth}, this inequality yields
\begin{equation}\label{bound_of_nabla_uk}
\int_{D_{j}} |\nabla u_{k}|^{p} \, d w
\le
C \int_{D_{j + 1}} |u_{k}|^{p} \, d w 
+
C \left\| \frac{f}{w} \right\|_{L^{\infty}(D_{j + 1})}^{p / (p - 1)} w(D_{j + 1}).
\end{equation}
By \eqref{eqn:bound_by_v} and assumption on $f$, the right-hand side is bounded with respect to $k$.
Meanwhile, Lemma \ref{lem:weak_harnack} yields a local H\"{o}lder estimate of $u_{k}$.
Therefore, for each $j \ge 1$, there are constants $C_{j}$ and $\alpha_{j} \in (0, 1)$ independent of $k$ such that
\[
\| u_{k} \|_{C^{\alpha_{j}}(\cl{D_{j}})}
+
\| \nabla u_{k} \|_{L^{p}(D_{j}, w)} \le C_{j}.
\]
Take a subsequence of $\{ u_{k} \}_{k = 1}^{\infty}$ and a function $u$ on $\Omega$ such that
$u_{k} \to u$ locally uniformly in $\Omega$ and $\nabla u_{k} \wkto \nabla u$ weakly in $L^{p}(D_{j}; w)$ for all $j$.
Let $\eta$ be the function defined before.
By the product rule, we have
\[
\int_{\Omega} |\nabla (u_{k} \eta)|^{p} \, dw
\le
C \left(
\int_{D_{j + 1}} |\nabla u_{k}|^{p} \eta^{p} \, dw
+
\int_{D_{j + 1}} |u_{k}|^{p} |\nabla \eta|^{p} \, dw
\right).
\]
The left-hand side belong to $H_{0}^{1, p}(\Omega; w)$, and thus, $u \eta \in H_{0}^{1, p}(\Omega; w)$.
This implies that the zero extension of $u$ belongs to $H^{1, p}_{\loc}(\R^{n} \setminus \Gamma; w)$.
Using the test function $u \eta - u_{k} \eta \in H_{0}^{1, p}(\Omega; w)$, we obtain
\begin{equation}\label{eqn:aeconv1}
\begin{split}
\int_{\Omega} \mathcal{A}(x, \nabla u_{k}) \cdot \nabla (u - u_{k}) \eta \, dx
 & =
- \int_{\Omega} \mathcal{A}(x, \nabla u_{k}) \cdot \nabla \eta (u - u_{k}) \, dx
\\
& \quad +
\int_{\Omega} f (u - u_{k}) \eta \, dx.
\end{split}
\end{equation}
Take the limit $k \to \infty$.
The latter term on the right-hand side goes to zero.
Moreover, by \eqref{eqn:growth}, the former term is estimated by
\begin{equation*}
\begin{split}
& 
\left| \int_{\Omega} \mathcal{A}(x, \nabla u_{k}) \cdot \nabla \eta (u - u_{k}) \, dx \right|
\\
& \quad
\le
C \left( \int_{D_{j + 1} }|\nabla u_{j}|^{p} \, dw \right)^{(p - 1)/ p}
\left( \int_{D_{j + 1}} |u - u_{k}|^{p} \, dw \right)^{1 / p}.
\end{split}
\end{equation*}
The right-hand side goes to zero because $u_{k} \to u$ uniformly in $D_{j + 1}$.
Thus, the left-hand side of \eqref{eqn:aeconv1} goes to zero.
Meanwhile, by the weak convergence of $\nabla u_{j}$ in $L^{p}(D_{j + 1}; w)$, we have
\begin{equation}\label{eqn:aeconv2}
\int_{\Omega} \mathcal{A}(x, \nabla u) \cdot \nabla (u - u_{k}) \eta \, dx \to 0.
\end{equation}
Combining \eqref{eqn:aeconv1}, \eqref{eqn:aeconv2} and \eqref{eqn:monotonicity}, we find that
\[
\int_{D_{j}} \mathcal{A}(x, \nabla u) - \mathcal{A}(x, \nabla u_{k}) \cdot \nabla (u - u_{k}) \eta \, dx \to 0.
\]
It follows from \cite[Lemma 3.73]{MR2305115}
that $u$ satisfies $- \divergence \mathcal{A}(x, \nabla u) = f$ in $\Omega$.
Interior regularity of $u$ follows from Lemma \ref{lem:weak_harnack}.
If $\xi \in \Gamma$, then $u$ is continuous at $\xi$ by the upper bound \eqref{eqn:bound_by_v}.

Let $u, v \in H^{1, p}_{\loc}(\Omega; w) \cap C(\Omega)$ be weak solutions to \eqref{eqn:p-laplace}
satisfying the Dirichlet boundary condition in the statement.
Assume that $u(x) > v(x)$ for some $x \in \Omega$.
Then, $D = \{ x \in \Omega \colon u(x) > v(x) + \epsilon \}$ is a nonempty bounded open set for some $\epsilon > 0$.
If $\dist(\cl{ D }, \Gamma) = 0$, then there is a boundary point $\xi \in \cl{D} \cap \Gamma$.
This contradicts to assumption because $(u - v)(\xi) = 0$ and $\inf_{D} (u - v) \ge \epsilon > 0$.
Therefore, $\dist(\cl{ D }, \Gamma) > 0$ and $u, v \in H^{1, p}(D; w)$.
By density,
\[
\int_{D} \left( \mathcal{A}(x, \nabla u) - \mathcal{A}(x, \nabla v) \right) \cdot \nabla \varphi \, dx = 0
\]
for all $\varphi \in H_{0}^{1, p}(D; w)$.
Testing this equation with $\varphi = u - v - \epsilon$ and using \eqref{eqn:monotonicity}, we find that $u = v$ on $D$.
This contradicts to the assumption.
\end{proof}

\begin{theorem}\label{thm:Dirichlet}
Assume that $\Omega$ is a bounded open set and that $\R^{n} \setminus \Omega$ is uniformly $(p, w)$-fat.
Let $f \in L^{\infty}_{\loc}(\Omega)$, and assume that there are constants $K > 0$ and $0 < \beta \le \alpha$ such that
\begin{equation*}\label{eqn:data}
|f(x)| \le K \delta(x)^{\beta(p - 1) - p} \, w(x)
\end{equation*}
for almost every $x \in \Omega$, where $\alpha$ is a positive number in Theorem \ref{thm:main}.
Then, there exists a unique weak solution $u \in H_{\loc}^{1, p}(\Omega; w) \cap C(\cl{\Omega})$ to \eqref{eqn:p-laplace}.
Moreover,
\begin{equation}\label{eqn:bound_of_u}
|u(x)| \le C K^{1 / (p - 1)}  \delta(x)^{\beta}
\end{equation}
for all $x \in \Omega$, where $C = c_{H}^{1 - p} 30^{\beta / \alpha} (\alpha / \beta)$.
\end{theorem}

\begin{proof}
Applying to Proposition \ref{prop:Dirichlet} to $\Gamma = \partial \Omega$ and $h(t) = (K / c_{H})^{1 / (p - 1)} t^{\beta / \alpha}$,
we obtain the desired unique weak solution $u$.
The upper bound \eqref{eqn:bound_of_u} follows from \eqref{eqn:bound_by_v}.
\end{proof}

\section*{Acknowledgments}
This work was supported by JST CREST (doi:10.13039/501100003382) Grant Number JPMJCR18K3
and JSPS KAKENHI (doi:10.13039/501100001691) Grant Number JP17H01092.


\bibliographystyle{abbrv} 
\bibliography{reference}


\end{document}